\documentclass[
final
]{amsart}

\usepackage[english]{babel}
\usepackage[utf8]{inputenc}

\usepackage[T1]{fontenc}
\usepackage{lmodern}
\makeatletter
\newsavebox{\@brx}
\newcommand{\llangle}[1][]{\savebox{\@brx}{\(\m@th{#1\langle}\)}%
  \mathopen{\copy\@brx\mkern2mu\kern-0.9\wd\@brx\usebox{\@brx}}}
\newcommand{\rrangle}[1][]{\savebox{\@brx}{\(\m@th{#1\rangle}\)}%
  \mathclose{\copy\@brx\mkern2mu\kern-0.9\wd\@brx\usebox{\@brx}}}
\makeatother

\usepackage[inline]{enumitem}
\setlist[enumerate,1]{label={(\roman*)}}
\usepackage{verbatim}
\usepackage{url}
\usepackage{xcolor}
\usepackage[notref,notcite,color]{showkeys}
\usepackage{microtype}
\usepackage{varwidth} 
\usepackage{mathtools}
\usepackage{url}

\usepackage[LogicalSystemsSans]{ak_logic}

\newcommand*{\Theorem}{Theorem}
\newcommand*{\Proposition}{Proposition}
\newcommand*{\Lemma}{Lemma}
\newcommand*{\Corollary}{Corollary}
\newcommand*{\Definition}{Definition}
\newcommand*{\Question}{Question}
\newcommand*{\Remark}{Remark}
\newcommand*{\Notation}{Notation}
\newcommand*{\Figure}{Figure}

\theoremstyle{plain}
\newtheorem{theorem}{\Theorem}
\newtheorem{proposition}[theorem]{\Proposition}
\newtheorem{corollary}[theorem]{\Corollary}

\theoremstyle{definition}

\theoremstyle{remark}
\newtheorem{remark}[theorem]{\Remark}

\newenvironment{claim}[1]{\smallskip\par\noindent\underline{Claim\space#1:}\space}{\smallskip}
\newenvironment{claimproof}[1]{\smallskip\par\noindent\underline{Proof of Claim\space#1:}\space}{\leavevmode\unskip\penalty9999 \hbox{}\nobreak\hfill\quad\hbox{$\blacksquare$}\smallskip}

\usepackage{prettyref}
\newrefformat{thm}{\Theorem~\ref{#1}}
\newrefformat{pro}{\Proposition~\ref{#1}}
\newrefformat{cor}{\Corollary~\ref{#1}}
\newrefformat{lem}{\Lemma~\ref{#1}}
\newrefformat{rem}{\Remark~\ref{#1}}
\newrefformat{def}{\Definition~\ref{#1}}
\newrefformat{fig}{\Figure~\ref{#1}}

\newcommand{\fin}[1][k]{\ensuremath{\textup{FIN}_{\mathnormal{#1}}}\xspace}
\newcommand{\supp}{\textup{supp}}

\title{A lower bound on Gowers' \fin theorem}
\author{Alexander P.\ Kreuzer}
\address{Department of Mathematics \\
Faculty of Science \\
National University of Singapore \\
Block S17, 10 Lower Kent Ridge Road \\
Singapore 119076 
}

\email{matkaps@nus.edu.sg}
\urladdr{\url{http://aleph.one/matkaps/}}
\usepackage{scrtime}
\date{\today\ \thistime}

\begin{document}
\maketitle

Gowers' \fin theorem, also called Gowers' pigeonhole principle or Gowers' theorem, is a Ramsey-type theorem. It first occurred in the study of Banach space theory \cite{wG92}, and
is a natural generalization of Hindman's theorem.
In this short note, we will show that Gowers' \fin theorem does not follow from \lp{ACA_0}.

\section{Hindman's theorem}

Hindman's theorem is the following statement. As the name suggest it was established by Neil Hindman, see \cite{nH74}.

\begin{theorem}[Hindman's theorem, \lp{HT}]
  If the natural numbers are colored with finitely many colors then there is an infinite set $S\subseteq \Nat$ such that the non-repeating, finite sums of $S$
  \[
  \textup{FS}(S) := \left\{\, \sum_{i\in I}  s_i\sizeMid  I \in \mathcal{P}_\textnormal{fin}(\Nat) \setminus \{ \emptyset \}\, \right\}
  \quad\text{where $(s_i)$ is the enumeration of $S$}
  \]
  are colored with only one color.
\end{theorem}

In the context of reverse mathematics \lp{HT} was first investigated by Blass, Hirst, Simpsons in \cite{BHS87}. 
There it was shown that it follows from  \ls{ACA_0^+},
that is \ls{ACA_0} plus the statement that for all $X$ the $\omega$-jump $X^{(\omega)}$ exists. The best known lower bound is \ls{ACA_0}, see also \cite{BHS87}.
It is one of the big open questions of reverse mathematics what is the exact strength of \lp{HT} and whether it is equivalent to \ls{ACA_0}.
There has been some partial process on this question. However no definite answer could be given. See \cite{jH04,aB05,hT12} and \cite[Section 2.3]{aM11}.

As already said Gower's pigeonhole principle is a generalization of \lp{HT}. Below we will see that it is not provable in \ls{ACA_0}.

To understand the statement of Gower's pigeonhole principle we will first look at the following finite unions variant of Hindman's theorem. 
It is not to difficult to see that it s equivalent (relative to \ls{RCA_0}) to \lp{HT}, see \cite{BHS87}.
\begin{theorem}[Hindman's theorem, finite unions variant]
  If the finite subsets of the natural numbers $\mathcal{P}_\textnormal{fin}(\Nat)$ are colored with finitely many colors there exists a infinite set $S\subseteq \mathcal{P}_\textnormal{fin}(\Nat)$ consisting of pairwise disjoint sets and   such that the non-empty finite unions of $S$
  \begin{equation}\label{eq:nu}
  \textup{NU}(S) := \{ s_1 \cup \dots \cup s_n \mid n\in\Nat\setminus \{0\}, s_i\in S, \max s_i < \min s_{i+1} \}
  \end{equation}
  are colored with only one color.
\end{theorem}

\section{Gowers' \fin theorem}
Before we can formulate Gowers' \fin theorem we have to introduce some notations. The following definitions will be made in \ls{RCA_0}.

Let $k\in \Nat$ and let $p\colon \Nat \longto [0,k]$.
We call the set the  
\[
\supp(p):= \{ n\in \Nat \mid p(n) \neq 0 \}
\]
the \emph{support} of $p$. 
The space \fin we be the following.
\begin{align*}
\fin &:= \left\{\, p\in {[0,k]}^\Nat \sizeMid \abs{\supp(p)} < \omega, \Exists{n} p(n)=k \,\right\} \\
  & \phantom{:}= \left\{\, p\in {[0,k]}^{<\Nat} \sizeMid \Exists{n< \lth(p)} (p(n)=k) \,\right\} 
\end{align*}
This space will play the role of $\mathcal{P}_\textnormal{fin}(\Nat)$ in Hindman's theorem. For $k=2$ it is actually isomorphic to $\mathcal{P}_\textnormal{fin}(\Nat)$.

On \fin we define the following order
\[
p < q \quad\text{if{f}}\quad \max \supp(p) < \min \supp(q)
\]
and the following partial addition for comparable elements
\[
p + q \quad \text{for }p<q
\]
which will be equal to the usual pointwise addition (if $p<q$). 

On \fin we will make use of the following, so called ``tetris'' operation $T$
\[
T\colon \fin \longto \fin[k-1] \qquad T(p)(n) = p(n) \dotminus 1
.\]

A \emph{block sequence} $B$ is an infinite increasing sequence $B=(b_n)_{n\in\Nat}$ in \fin. The \emph{combinatorial space} $\langle B \rangle$ generated by $B$ is the smallest subsets of \fin containing $B$ and closed under addition and tetris, i.e.,
\[
\langle B \rangle := \left\{ \sum_{n\in \Nat} T^{k-f(n)}(b_n) \mid f\in \fin \right\}
.\]
(Note that the above sum is finite since the support of $f$ is finite.)

We can now formulate Gowers' \fin theorem.

\begin{theorem}[Gowers' \fin theorem, \lp{FIN_{<\infty}}]
  For any $k\in\Nat$ and any finite coloring of $\fin$ there exists a combinatorial subspace colored by only one color.
\end{theorem}

We will denote full version of of Gowers' \fin theorem by \lp{FIN_{<\infty}} and the restriction to a particular $k$ by \lp{FIN_\mathnormal{k}}. 
It is clear that for $k=2$ this theorem is the same as Hindman's theorem since a combinatorial subspace in \fin[2] is just as the set given in \eqref{eq:nu}. So in other words
\[
\ls{RCA_0} \vdash \lp{HT} \IFF \lp{FIN_\mathnormal{2}}
.\]
Moreover, it is clear that $\lp{FIN_\mathnormal{k}} \IMPL \lp{FIN_\mathnormal{l}}$ if $k>l$ since \fin[l] can be embedded into \fin via the following mapping $i(p)(n) = p(n) + k-l$.

\begin{remark}[\ls{RCA_0}]
  A combinatorial  subspace $\langle B \rangle \subseteq \fin$ is isomorphic to \fin via the isomorphism $\Theta_B(f) :=  \sum_{n\in \Nat} T^{k-f(n)}(b_n)$.

  Thus, \lp{FIN_\mathnormal{k}} remains true if one colors only a combinatorial subspace instead of \fin.
  This variant is equivalent to \lp{FIN_\mathnormal{k}}.
\end{remark}

\section{A lower bound on \lp{FIN_\mathnormal{k}}}

\begin{theorem}\label{thm:low}
  There exists a recursive coloring of \fin[k+1] with $2^k$ many colors, such that each monochromatic combinatorial subspace computes $\emptyset^{(k)}$.
\end{theorem}
Before we will come to the proof we will fix some notation and state a proposition.
We will need computable approximations of the $n$-fold Turing jump. For this we shall write
\[
\emptyset^{(n)}_{s_n,s_{n-1},\dots,s_1}
\]
for 
\[
(\cdots (\emptyset'_{s_1})'_{s_2} \dots )'_{s_n}
\]
(In other words the $s_n$-step approximation of the Turing jump of the $s_{n-1}$-step approximation of the Turing jump \dots\ of the $s_1$-step approximation of the unrelativized Turing jump.)
If the tuple $(s_n,s_{n-1},\dots,s_m)$ is shorter than $n$, we shall take the true Turing jump for the missing indexes.

\begin{proposition}\label{pro:itlim}
  For each $n,m$ there exists a finite sequence $(m_1,\dots,m_n)$ such that
  \[
  \emptyset^{(n)} \cap [0;m] = \emptyset^{(n)}_{m_n,\dots,m_1} \cap [0;m]
  .\]
  Moreover, we may assume that $(m_1,\dots,m_n)$ is such that, taking $m_{n+1}:=m$,
  \begin{equation}\label{eq:lim}
  \emptyset^{(i)} \cap [0;m_{i+1}] = \big(\emptyset^{(i-1)}\big)'_{s} \cap [0;m_{i+1}] \quad\text{for $s>m_i$}
  \end{equation}
  where $i\in [1;n]$.
\end{proposition}
\begin{proof}
  Apply the limit lemma---relative to $\emptyset^{(n-1)}$---we obtain an $m_n$ such that 
  \[
  \emptyset^{(n)} \cap [0; m_{n+1}] = \big(\emptyset^{(n-1)})'_s \quad \text{for all $s>m_n$}
  .\]
  Iterating this process we obtain $(m_1,\dots,m_n)$ satisfying \eqref{eq:lim}. This sequence automatically satisfies the other statement of the proposition.
\end{proof}

For an $f\in \fin[k]$ with $k\ge 2$ we shall write
$\mu_i(f):= \max \{ n \mid f(n) =i\}$ and $\lambda_i(f):= \min \{ n \mid f(n) =i\}$ and $\mu(f):= \max\{ n \mid f(n) \neq 0 \}$ and $\lambda(f) :=\min \{ n \mid f(n) \neq 0\}$. Note that $\mu_i,\lambda_i$ are undefined if $i$ is not in the image of $f$. However $\mu,\lambda$ is by definition of $\fin[k]$ always defined.

\begin{proof}[Proof of \prettyref{thm:low}]
  This proof is inspired by Theorem~2.2 of \cite{BHS87}.

  Let $f\in \bigcup_{k'\le k}\fin[k'+1]$. Fix an $i\ge 1$ and let $(n_0,\dots,n_l)$ be the indexes (in ascending order) where $f(n_j)=i$.
  We call $(n_j,n_{j+1})$ a \emph{short gap$_i$} if
  \[
  \emptyset^{(i)} \cap [0;n_j]\; \neq\; \big(\emptyset^{(i-1)}\big)'_{n_{j+1}} \cap [0;n_j]
  .\]
  We will write $\textsf{SG}_{\mathnormal{i}}(f)$ for the number of short gaps$_{i}$ in $f$. Note that in general $\textsf{SG}_\mathnormal{i}(f)$ is not computable.

  We will now construct computable approximation of short gaps.
  Let $i$, $(n_0,\dots,n_l)$ as above. We call $(n_j,n_{j+1})$ a \emph{very short gap$_i$} if 
  \begin{equation}\label{eq:vsg}
    \emptyset^{(i)}_{\mu_i(f),\mu_{i-1}(f),\dots,\mu_1(f)} \cap [0;n_j]\; \neq\; \emptyset^{(i)}_{\mu_i(f),\mu_{i-1}(f),\dots,\mu_{2}(f),n_{j+1}} \cap [0;n_j].
  \end{equation}
  (We treat $\mu_i(f)$ as if it were $0$ if it is undefined.)
  E.g.~for $i=1$ we call $(n_j,n_{j+1})$ a very short gap$_{1}$ if
  \[
  \emptyset'_{\mu_1(f)} \cap [0;n_j]\; \neq\; \emptyset'_{n_{j+1}}\cap [0;n_j]
  .\]
  Note that $\textsf{VSG}_i(f)$ is computable.

  We color $\fin[k+1]$ with the following coloring.
  \[
  c(f) := \sum_{i=1}^{k} 2^{i-1} \cdot (\textsf{VSG}_i(f) \bmod 2).  
  \]
  By \lp{FIN_{k+1}} there  is a homogeneous combinatorial subspace $B$.

  We will write $\llangle B \rrangle$ for $\bigcup_{i=0}^{k} T^i \langle B \rangle$.

  \begin{claim}{1}
    For each $f\in \llangle B\rrangle$ there exists $g\in \langle B \rangle$ with $f<g$ such that every short gap$_i$ in $f$ is a very short gap$_i$ in $f+g$ and such that between $f$ and $g$ no gap$_i$ is (very) short gap$_i$ and not very short.
  \end{claim}
  \begin{claimproof}{1}
    Recursively build a sequence $(g_i)_{i \le k}\subseteq \langle B \rangle$ with
    \begin{enumerate}
    \item $g_0 > f$ and $g_{i+1} > g_{i}$,
    \item\label{enum:1:3} $\emptyset^{(k-i')} \cap [0,\mu(g_i)] = (\emptyset^{(k-i'-1)})'_{s}\cap [0,\mu(g_i)]$ \\ for $i' \le i< k$ and $s \ge \lambda(g_{i+1})$.
    \end{enumerate}
    The proof proceeds in the same was as the proof of \prettyref{pro:itlim}.
    Suppose we have chosen $g_{0},\dots,g_{i-1}$, then by the limit lemma there exists an $m$ such that the equation in \ref{enum:1:3} is true for all $s \ge m$. Choose $g_i$ such that $g_i>g_{i-1}$ and such that $\lambda(g_i) > s'$.
    Setting
    \[
    m_{k-i} := \mu_{k+1}(g_i) \qquad\text{for } i < k
    \]
    we recovers a sequence $(m_1,\dots,m_k)$ as in \prettyref{pro:itlim}.
    From this we get in the same way that
    \[
    \emptyset^{(i)}_{m_{i},\dots,m_{1}} \cap [0,m_{i+1}] = \emptyset^{(i)} \cap [0,m_{i+1}]
    .\] 
    Now consider
    \[
    g:= \sum_{i=0}^k T^i(g_i)
    .\]
    By definition we have that $\mu_{k-i}(g) = \mu_{k-i}(f+g) = m_{k-i}$. Therefore, the right hand side of \eqref{eq:vsg} for $f+g$ is equal $\emptyset^{(i)}$ for $[0;n_j] \subseteq [0;\mu(f)]$. With this the claim is satisfied.
  \end{claimproof}

  \begin{claim}{2}
    $\textsf{SG}_i(f)$ is even for each $f\in \llangle B \rrangle$.
  \end{claim}
  \begin{claimproof}{2}
    Assume that $f\in \llangle B \rrangle$ and take again $g$ as in Claim~1. We get
    \begin{equation*}
      \textsf{VSG}_i(f+g) = \textsf{SG}_i(f) + \textsf{VSG}_i(g)
      .
    \end{equation*}
    Since $f+g,g\in \langle B \rangle$, the parity of $\textsf{VSG}_i(f+g),  \textsf{VSG}_i(g)$ is the same by assumption to $B$. Therefore, $\textsf{SG}_i(f)$ must be even.
  \end{claimproof}

  We now show by induction that one can compute $\emptyset^{(i)}$ for $i\le k$ from $B$. Assume that we already have an algorithm which computes $\emptyset^{(i-1)}$. To compute whether $x$ is contained in $\emptyset^{(i)}$ or not search for an $f,g\in \langle B \rangle$ with
  \begin{enumerate}
  \item $f<g$,
  \item $x <\lambda(f)$, and
  \item\label{enum:2:3} the image of $f,g$ contains $i$ (this can always be achieved by searching for $f_1<f_2\in \langle B \rangle$ and taking $f:= f_1 + T^{k+1-i}(f_2)$).
  \end{enumerate}
  By Claim~2, we know that $\textsf{SG}_i(f),\textsf{SG}_i(g), \textsf{SG}_i(f+g)$ is even. Thus, $(\mu_i(f),\lambda_i(g))$ is not a short gap$_i$. (For this argument we use \ref{enum:2:3} and the consequence that $\mu_i(f),\lambda_i(g)$ are defined.)
  Therefore,
  \[
  x\in \emptyset^{(i)} \quad\text{if{f}}\quad
  x\in {\big(\emptyset^{(i-1)}\big)}'_{\lambda_i(g)}
  \]
  By induction hypothesis $\emptyset^{(i-1)}$ is computable relative to $B$. Therefore $x\in \emptyset^{(i)}$ is computable in $B$, too.
\end{proof}

\begin{remark}\label{rem:corl}
  The above proof formalizes in $I\Sigma^0_{\mathnormal{k+2}}$. The critical steps where induction is used are Claim~1, and the verification of the algorithm. In Claim~1, $B\Sigma^0_2$ relative to $\emptyset^{(i)}$ for $i<k$ is used to find the $m$. The analysis of this the same as in the limit lemma and it is equivalent to $B\Sigma_2$. In the verification of the algorithm we have to show that, calling the algorithm for the $i$-Turing jump $e_i$, that
  \[
  \Forall{x} \left(\Phi_{e_i}^B(x)= 1 \IFF x\in \emptyset^{(i)}\right)
  \]
  for all $i\le k$. Since this statement is $\Pi^0_{k+2}$, $I\Sigma^0_{k+2}$ is sufficient.
\end{remark}

\begin{corollary}\mbox{}
  \begin{enumerate}
  \item\label{enum:corl:1} For all $k$ we have that
    $\ls{RCA_0} + \lp{FIN_\mathnormal{k}} + \lp[\Sigma_\mathnormal{k+2}]{IND}$ proves that $\emptyset^{(k)}$ exists.
  \item\label{enum:corl:2} $\ls{ACA_0} + \lp[\Delta^1_1]{IND} + \lp{FIN_{<\infty}} \vdash \Forall{X}\Forall{k} X^{(k)} \text{ exists}$. In other words, the above theory proves \ls{ACA_0'}.
  \item\label{enum:corl:3} $\ls{ACA_0} \nvdash \lp{FIN_{<\infty}}$.
  \end{enumerate}
\end{corollary}
\begin{proof}
  \ref{enum:corl:1} is just a reformulation of \prettyref{rem:corl}. \ref{enum:corl:2} follows from \ref{enum:corl:1} by noting that \lp[\Delta^1_1]{IND} implies \lp[\Sigma_\mathnormal{n}]{IND} uniformly for each $n$ using Skolemization.

  \noindent
  \ref{enum:corl:3} follows from the following. First, $\Forall{X,k} X^{(k)} \text{ exists}$ can be written as a $\Pi^1_2$\nobreakdash-\hspace{0pt}statement, i.e.,
  \[
  \Forall{X,k} \Exists{Y} \left( Y_0 = X \AND \Forall{i< k} Y_{i+1} = \text{TJ}(Y_{i})\right)
  .\]
  Now if $\lp{FIN_{<\infty}}$ would be provable in \ls{ACA_0} then the above statement would be provable already in $\ls{ACA_0} + \lp[\Delta^1_1]{IND}$ and a fortiori in $\ls{ACA_0} + \lp[\Delta^1_1]{CA}$. However, this theory is $\Pi^1_2$-conservative over \ls{ACA_0}, see \cite[IX.4.4]{sS09}, which leads to the contradiction $\ls{ACA_0} \vdash \ls{ACA_0'}$.
\end{proof}

\section{Conclusion}

We could show that the generalization of Hindman's theorem (\lp{HT}), Gowers' $\fin$ theorem ($\lp{FIN_{<\infty}}$) is stronger than the best known lower bound for \lp{HT}.
It remains open to find a matching upper bound for  $\lp{FIN_{<\infty}}$. It seems to be in general very difficult since to the knowledge of the author any known proofs of  $\lp{FIN_{<\infty}}$ makes use of special ultrafilters (or similar objects). Of course by Shoenfield absoluteness  $\lp{FIN_{<\infty}}$ must be provable without the axiom of choice.

\bibliographystyle{amsplain}
\bibliography{../bib}

\end{document}